\documentclass[12pt]{amsart}%
\usepackage[latin9]{inputenc}
\usepackage{amsmath}
\usepackage{amssymb}
\usepackage{amsfonts}
\usepackage{amsthm}
\usepackage{latexsym}
\usepackage{graphicx}
\usepackage{cite}

\makeatletter
\newtheorem{theorem}{Theorem}

\newtheorem{remark}[theorem]{Remark}

\newtheorem{proposition}[theorem]{Proposition}

\begin{document}
\title[Range description for the free space wave operator]{Half-time Range description for the free space wave operator and the spherical
means transform}
\author{P.~Kuchment$^1$ and L.~Kunyansky$^2$}

\address{$^1$ Department of Mathematics,
Texas A\&M University,
College Station, TX 77843, USA.
}

\address{$^2$ Department of Mathematics,
University of Arizona,
Tucson, AZ 85721, USA.
}

\begin{abstract}
The forward problem arising in several hybrid imaging modalities
can be modeled by the Cauchy problem for the free space wave equation.
Solution to this problems describes propagation of a pressure wave, generated
by a source supported inside unit sphere $S$. The data $g$ represent the time-dependent
values of the pressure on the observation surface $S$.  Finding initial pressure $f$ from the known
values of $g$ consitutes the inverse problem. The latter is also  frequently
formulated in terms of the spherical means of $f$ with centers on~$S$.

Here we consider a problem of range description of the wave operator mapping $f$ into $g$.
Such a problem was considered before, with data $g$ known on time interval at least
$[0,2]$ (assuming the unit speed of sound). Range conditions were also found in terms of
spherical means, with radii of integration spheres lying in the range $[0,2]$.
However, such data are redundant. We present necessary and sufficient conditions
for function $g$ to be in the range of the wave operator, for $g$ given on a half-time
interval $[0,1]$. This also implies range conditions on spherical means measured
for the radii in the range $[0,1]$.

\end{abstract}
\maketitle

\section{Introduction}\label{S:intro}

In the thermo- and photoacoustic tomography (TAT/PAT)
\cite{KrugerLFA-MP-95,KrugerRK-MP-99,Oraev94} the forward problem can be
modeled (in the simplest case) by the following Cauchy problem for the whole
space wave equation:%
\begin{align}
u_{tt}(t,x)  &  =\Delta u(t,x),\quad(t,x)\in(0,\infty)\times\mathbb{R}%
^{d},\label{E:originalwave}\\
u(0,x)  &  =f(x),\qquad u_{t}(0,x)=0. \label{E:originalBC}%
\end{align}
Solution $u(t,x)$ of this problem represents the excess pressure in the
acoustic wave. This wave is caused by an instantaneous thermoelastic expansion
at $t=0$ leading to the initial condition $u(0,x)=f(x).$ Function $f(x)$ is
assumed to be supported within the open unit ball $B\subset\mathbb{R}^{d}$
centered at the origin, and is extended by zero to $\mathbb{R}^{d}$. The
measurements in TAT/PAT are the values $g(t,\theta)$ of the pressure
registered during the time interval $t\in(0,T]$ on the unit sphere $\theta\in
S:=\mathbb{S}^{d-1}$:%
\begin{equation}
g(t,\theta):=u(t,\theta),\quad(t,\theta)\in Z_{T}, \label{E:measurement}%
\end{equation}
where $Z_{T}:=(0,T]\times S$ is the cylinder supporting the data. The data
$g(t,\theta)$ can be viewed as an action of \textbf{the observation operator}
$\mathcal{A}$ on function $f$ :
\begin{equation}
\mathcal{A}:f(x)\in C_{0}^{\infty}(B)\mapsto g(t,\theta),\quad(t,\theta)\in
Z_{T}. \label{E:operator}%
\end{equation}

The inverse source problem of TAT/PAT consists of finding $f$ from data
$g=\mathcal{A}f$. It has been studied extensively in the last two decades (see
\cite{KuKuHand15,KuchCBMS14,KuchScherzer,KuKuEncycl} and references therein).
In particular, explicit inversion formulas have been found \cite{MXW1,Finch04,
Finch07,KunyanskyIP07} allowing one to invert the operator $\mathcal{A}$ and
reconstruct $f$ from $g$ under the assumption $T\geq2$. In this case, $T=2$
represents the time sufficient for all characteristics originated within $B$
to reach the boundary $\partial B$ of $B$. Thus, the most studied case is
$T\geq2$. However, shorter acquisition times are of practical and theoretical
importance. We concentrate here on the case $T=1$, which is known to be
sufficient, since by that time at least one of two characteristics originating
at any point in $B$ reaches the boundary $\partial B$. In addition, explicit
inversion formulas have been found \cite{AmbGouia18,Do-Kun18} for
reconstructing $f$ from \textquotedblleft half-time\textquotedblright\ data
$g$ known on the cylinder $Z_{1}:=(0,1]\times S$.

It is well known \cite{JohnSph04} that the data $g(t,\theta)$ obtained by
solving the problem (\ref{E:originalwave})-(\ref{E:measurement}) can be
expressed in terms of the spherical means $M(r,\theta)$ of $f$ with centers
$\theta\in$ $S$:%
\begin{equation}
M(r,\theta):=\int\limits_{S}f(\theta+r\tau)d\tau,\quad(r,\theta)\in
\mathbb{R}^{+}\times S, \label{E:means}%
\end{equation}
where $d\tau$ is the standard area element on the unit sphere. Indeed,
\begin{equation}
g(t,\theta)=\frac{\partial}{\partial t}\int\limits_{0}^{t}M(r,\theta
)G_{d}(t,r)r^{d-1}dr, \label{E:g_from_means}%
\end{equation}
where $\mathbf{G}(t,x)=G_{d}(t,|x|)$ is the causal free-space Green's function
for the wave equation (\ref{E:originalwave}). Explicit expressions for
$G_{d}(t,r)$ are well known (see, e.g., \cite{variations}):
\begin{equation}%
\begin{array}
[c]{c}%
G_{2}(t,r)=\frac{\chi_{+}(t-r)}{2\pi\sqrt{t^{2}-r^{2}}},\quad G_{3}%
(t,r)=\frac{\delta(t-r)}{4\pi r},\\
\\
\mbox{ and for }d>3,\\
\\
G_{d+2}(t,r)=-\frac{1}{2\pi r}\frac{\partial}{\partial r}G_{d}%
(t,r),\label{E:Green-new}%
\end{array}
\end{equation}
Here $\delta(t)$ is the Dirac's delta function, and $\chi_{+}(s)$ is the
Heaviside function, equal to 1 when $s>0$ and equal to 0 otherwise.

Equation (\ref{E:g_from_means}) is easily solvable with respect to
$M(r,\theta)$, so that the knowledge of $g(t,\theta)$ on $Z_{T}$ implies the
knowledge of $M(r,\theta)$ on $Z_{T}$ and \emph{vice versa} (see, e.g.
\cite{AgrFK09,JohnSph04}). As a consequence of this, historically, most of the
results in this area have been obtained in terms of the spherical means $M$,
rather than in terms of the wave data $g$.

We are interested here in the range of the operator $\mathcal{A}$ (see
(\ref{E:operator})). Namely, we study the necessary and sufficient conditions
for a function $g\in C_{0}^{\infty}(Z_{T})$ to lie in the range of
$\mathcal{A}$. Such range descriptions were obtained in
\cite{AgrKQ07,AgrFK09,finch2006range,venky2023range} in terms of spherical
means $M(r,\theta)$ defined on $C_{0}^{\infty}([0,2]\times S)$, which
corresponds to the case of \textquotedblleft full time\textquotedblright%
\ observation, i.e., $T=2$.

However, these results are not optimal, in the sense that they use
\textquotedblleft too much\textquotedblright\ data. Indeed, as we have already
mentioned, it is known that a stable solution of the inverse source problem of
TAT/PAT is possible if data $g(t,\theta)$ are given on a twice shorter
cylinder $Z_{1}$ (see, e.g., \cite{AmbGouia18,Do-Kun18}). To put it
differently, this implies that the data on the time interval $[1,\infty)$ are
uniquely determined by the data on interval $(0,1].$ In terms of applications
to TAT/PAT, this means that the data acquisition time can be halved
\cite{Anashalf}. This makes the measurement system cheaper, and potentially
improves the quality of the image, since, for a variety of reasons, the
acoustic wave quickly deteriorates in time, and thus \textquotedblleft
late\textquotedblright\ data get degraded. We refer to such data as
\textquotedblleft half-time\textquotedblright\ data, as is done in the title
of this paper.

In the present paper, we find the range description of the operator
\[
\mathcal{A}:C_{0}^{\infty}(B)\rightarrow C_{0}^{\infty}(Z_{1}),
\]
i.e. we are working with the data given on time interval $(0,1]$. Formulation
of the problem and the main result are given in Section \ref{S:statement},
some preliminary considerations are made in Section \ref{S:prelim} after which the main
results are described in Sections \ref{S:rangeRadon terms} and \ref{S:range_expansions}. The proofs are presented in
Section \ref{S:proofs}.



\section{Formulation of the problem}
\label{S:statement}


We consider operator
\begin{align}
\mathcal{A}  &  :C_{0}^{\infty}(B)\rightarrow C_{0}^{\infty}(Z_{1}%
),\nonumber\\
\mathcal{A}  &  :f\mapsto g(t,\theta),\quad(t,\theta)\in Z_{1},
\label{E:rangeA}%
\end{align}
where $g(t,\theta)$ is the trace of the solution of the Cauchy problem
(\ref{E:originalwave})-(\ref{E:originalBC}) on $Z_{1}$ (see
(\ref{E:measurement})). Our aim is to find necessary and sufficient condition
for a function $b(t,\theta)\in C_{0}^{\infty}(Z_{1})$ to be in the range of
$\mathcal{A}$, i.e. to be represented as $b=\mathcal{A}f$ for some $f\in
C_{0}^{\infty}(B)$. We somewhat abuse the notations here, since functions from
$C_{0}^{\infty}(Z_{1})$ and their derivatives are not required to vanish at
$t=1$. They vanish only near $t=0$.


\section{Some preliminary constructions}
\label{S:prelim}
Our approach is based on studying the Radon transform of a solution to a
certain auxiliary initial/boundary value problem (IBVP) for the wave equation
in the exterior of the closed ball $\overline{B}$. We start by defining and
analyzing the tools we need.

\subsection{The Radon transform and its range}
\label{SS:Radon}

Consider a compactly supported continuous function $q(x)$ on $\mathbb{R}^{d}$.
For a unit direction vector $\omega\in S$ and $p\in\mathbb{R}$, the values of
the Radon transform $\mathcal{R}q(\omega,p)$ of $q$ are given by the following
formula (see, e.g., \cite{Natt_old}):
\begin{equation}
\label{E:Radontr}%
\begin{array}
[c]{c}%
\left[  \mathcal{R}q\right]  (\omega,p):=\int\limits_{\mathbb{R}^{d}%
}q(x)\delta(x\cdot\omega-p)dx\\
=\int\limits_{\Pi(\omega,p)}q(x)ds \mbox{ for } (\omega,p)\in S\times
\mathbb{R}.
\end{array}
\end{equation}
Here $\delta(\cdot)$ is the Dirac's delta function and $\Pi(\omega,p)$ is the
plane defined by the equation
\begin{equation}
x\cdot\omega=p. \label{E:plane}%
\end{equation}

The following description of the range of the Radon transform $\mathcal{R}$ of
functions from $C_{0}^{\infty}(B)$ is well known
\cite{Helgason10,Natt_old,KuchCBMS14,GGG80,GGV,GGGbook03,Helg-RT-99,Ludwig}:

\begin{theorem}
\label{T:radon-range} A function $F(\omega,p)$ defined on $S\times
\mathbb{(}-1,1)$ can be represented as the Radon transform of a function $f\in
C_{0}^{\infty}(B)$, if and only if the following conditions are satisfied:

\begin{enumerate}
\item Symmetry condition: $F(\omega,p)=F(-\omega,-p)$,

\item Smoothness and support condition: $F(\omega,p)\in\mathbb{C}_{0}^{\infty
}(S\times(-1,1))$ ,

\item Moment conditions: for any $n=0,1,2,...$, the moment $M_{n}(\omega)$
\begin{equation}
M_{n}(\omega):=\int\limits_{-1}^{1}F(\omega,p)p^{n}dp \label{E:moments}%
\end{equation}
\newline is the restriction from $\mathbb{R}^{d}$ to $S$ of a homogeneous
polynomial of degree $n$ in $\omega$.
\end{enumerate}
\end{theorem}

We will use the real-valued spherical harmonics $Y_{l}^{\mathbf{m}}(\omega)$
on the unit $(d-1)$-dimensional sphere $S$, where $l=0,1,2,\dots$ and
$\left\vert \mathbf{m}\right\vert \leq l$ (see, for instance,
\cite{OlverNIST10}). They form an orthonormal basis in $L_{2}(S)$:
\begin{equation}
\int\limits_{S}Y_{l}^{\mathbf{m}}(\omega)Y_{l^{\prime}}^{\mathbf{m}^{\prime}%
}(\omega)d\omega=1~\text{if }l=l^{\prime},\mathbf{m=m}^{\prime}\text{, and
zero otherwise.} \label{E:Y}%
\end{equation}
Using these, the moment conditions (3) of the theorem can be equivalently
re-written as%
\begin{equation}%
\begin{array}
[c]{c}%
\int\limits_{S}M_{n}(\omega)Y_{l}^{\mathbf{m}}(\omega)d\omega=\int%
\limits_{S}\left[  \int\limits_{-1}^{1}F(\omega,p)p^{n}dp\right]
Y_{l}^{\mathbf{m}}(\omega)d\omega=0\\
\mbox{ for }n=0,1,2,...,l>n,\text{ }\left\vert \mathbf{m}\right\vert \leq
l.\label{E:standard_moment}%
\end{array}
\end{equation}

\subsubsection{Exterior Radon transform}

We will also need the \textbf{exterior (with respect to $\overline{B}$) Radon
transform} $\mathcal{R}^{E}$. For a continuous function $h(x)$ defined and
compactly supported in the exterior $B^{c}$ of $B$, $\mathcal{R}^{E}$ is
defined by the formula similar to (\ref{E:Radontr}) but with $|p| \ge1$:
\begin{equation}
\label{E:Radon_ext}%
\begin{array}
[c]{c}%
\left[  \mathcal{R}^{E}h\right]  (\omega,p):=\int\limits_{\mathbb{R}^{d}%
}h(x)\delta(x\cdot\omega-p)dx\\
=\int\limits_{\Pi(\omega,p)}h(x)ds, \mbox{ for }(\omega,p)\in S\times
(\mathbb{R}\setminus(-1,1)).
\end{array}
\end{equation}

\begin{remark}
If function $h$ were defined in the whole $\mathbb{R}^{d}$, not just outside
of the \textquotedblleft hole\textquotedblright\ $B$, transform $\mathcal{R}%
^{E}$ would be just a restriction of the full Radon transform of $h$.
\end{remark}

Our goal is to obtain range description for the wave operator,
similar to the above range description for the Radon transform $\mathcal{R}$.
We will achieve this by relating the Radon transform to a solution of a
certain IBVP (initial/boundary value problem) for the wave equation in the
exterior $B^{c}:=\mathbb{R}^{d}\setminus\overline{B}$ of the ball
$\overline{B}$. This technique was previously used in \cite{EllerKun} to
reduce the problem of inverting operator $\mathcal{A}$ to inverting the
classical Radon transform.

\subsection{Exterior initial/boundary value problem}

\label{SS:exterior}

We will denote by $F(\omega,p)$ the Radon transform of the initial function
$f(x)\in C_{0}^{\infty}(B)$:%
\begin{equation}
F(\omega,p):=\left[  \mathcal{R}f\right]  (\omega,p).
\label{E:Radon_projections}%
\end{equation}

Further, given a function $b\in C_{0}^{\infty}(Z_{1}),$ consider the solution
$v^{(b)}(t,x)$ to the following exterior problem in $(0,\infty)\times
\mathbb{R}^{d}\backslash\bar{B}$:
\begin{equation}
\left\{
\begin{array}
[c]{ll}%
\frac{\partial^{2}v^{(b)}(t,x)}{\partial t^{2}}-\Delta v^{(b)}(t,x)=0, &
(t,x)\in(0,1]\times B^{c},\\
v^{(b)}(0,x)=0,\quad v_{t}^{(b)}(0,x)=0, & x\in B^{c},\\
v^{(b)}(t,\theta)=b(t,\theta). & (t,\theta)\in Z_{1}.
\end{array}
\right.  \label{E:exterior_problem}%
\end{equation}
The domain $(0,1)\times B^{c}$ where the problem is solved, and the support of
the solution $v^{(b)}(t,x)$ are shown in Figure~1.
\begin{figure}[ht!]
\begin{center}%
\begin{tabular}
[c]{ccc}%
\includegraphics[scale = 0.6]{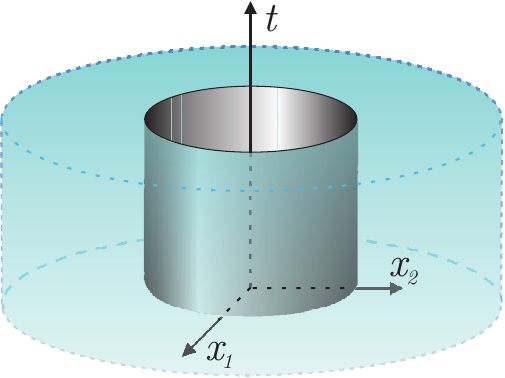} & \phantom{aa} &
\includegraphics[scale = 0.6]{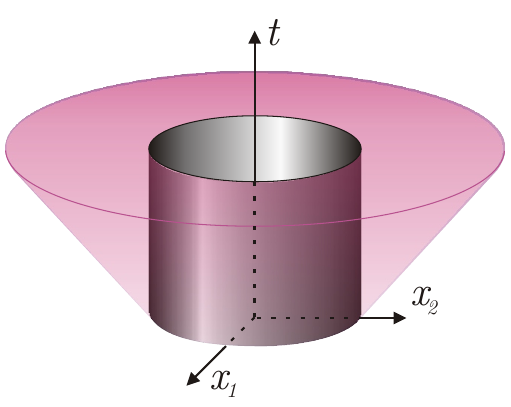}\\
(a) &  & (b)\\
&  &
\end{tabular}
\end{center}
\caption{Geometry of the exterior problem: \newline\phantom{aaa} (a) Domain
$(0,1)\times B^{c}$ where the problem is solved; \newline\phantom{aaa} (b)
Support of the solution $v^{(b)}(t,x)$.}%
\end{figure}

Note that the solution $u$ of the problem (\ref{E:originalwave}),
(\ref{E:originalBC}) also solves the exterior problem
(\ref{E:exterior_problem}), if $b(t,\theta)$ is equal to $g(t,\theta)$.
Therefore, if $b=g,$ the solution $v^{(g)}$ of the exterior problem
(\ref{E:exterior_problem}) coincides with $u$ in $(0,1]\times B^{c}.$ Let us
now consider for each fixed $t$ the Radon projections $\left[  \mathcal{R}u\right]  (t,\omega,p)$ (defined by
(\ref{E:Radontr})) of $u$. It is
well known (e.g., see \cite{JohnSph04}) that, for a fixed $\omega,$ such
projections satisfy the 1D wave equation:%
\begin{equation}
\frac{\partial^{2}}{\partial t^{2}}\left[  \mathcal{R}u\right]  (t,\omega
,p)=\frac{\partial^{2}}{\partial p^{2}}\left[  \mathcal{R}u\right]
(t,\omega,p),\qquad(t,p)\in(0,1)\times\mathbb{R}. \label{E:1Dwave}%
\end{equation}
Due to (\ref{E:originalBC}), $\mathcal{R}u$ satisfies the following initial conditions:%
\begin{equation}
\left[  \mathcal{R}u\right]  (0,\omega,p)=F(\omega,p),\quad\frac{\partial
}{\partial t}\left[  \mathcal{R}u\right]  (0,\omega,p)=0,\qquad p\in
\mathbb{R}. \label{E:1Dinitial}%
\end{equation}
According to d'Alembert's formula, one can express $\mathcal{R}u$ on
$(0,1)\times\mathbb{R}$ as follows:
\begin{equation}
\label{E:Radon_of_wave}%
\begin{array}
[c]{c}%
\lbrack\mathcal{R}u](t,\omega,p)=\frac{1}{2}\left[  F(\omega,p+t)+F(\omega
,p-t)\right] \\
\\
\mbox{ for } (t,p)\in(0,1)\times\mathbb{R}.
\end{array}
\end{equation}
Since $f(x)$ is compactly supported in $B$, its Radon projections
$F(\omega,p)$ are supported inside the region $\left\{  (\omega,p)\,:\,p\in
\lbrack-1,1]\right\}  $. Therefore, for $t=1$ we get
\begin{equation}
\lbrack\mathcal{R}u](1,\omega,p)=\left\{
\begin{array}
[c]{cc}%
\frac{1}{2}F(\omega,p-1), & p\in\lbrack1,2],\\
\frac{1}{2}F(\omega,p+1), & p\in\lbrack-2,-1]
\end{array}
\right.  , \label{E:Ru_from_f}%
\end{equation}
or
\[
F(\omega,p)=\left\{
\begin{array}
[c]{cc}%
2[\mathcal{R}u](1,\omega,p+1), & p\in\lbrack0,1]\\
2[\mathcal{R}u](1,\omega,p-1), & p\in\lbrack-1,0]
\end{array}
\right.  .
\]
Equivalently, if $v^{(g)}$ is the solution of the exterior problem
(\ref{E:exterior_problem}) with\ $b=g=\mathcal{A}f$, the Radon transform $F$
of $f $ can be expressed as follows:
\begin{equation}
F(\omega,p)=\left\{
\begin{array}
[c]{cc}%
2[\mathcal{R}v^{(g)}](1,\omega,p+1), & p\in\lbrack0,1]\\
2[\mathcal{R}v^{(g)}](1,\omega,p-1), & p\in\lbrack-1,0]
\end{array}
\right.  . \label{E:Rvg}%
\end{equation}

We thus have proven the following theorem (that first appeared in
\cite{EllerKun}):

\begin{theorem}
The Radon transform $F(\omega,p)$ of the initial function $f(x)$, and thus the
function itself, can be recovered by using (\ref{E:Rvg}) from the solution $v$
of the exterior problem (\ref{E:exterior_problem}).
\end{theorem}


\section{ The implicit range characterization theorem}
\label{S:rangeRadon terms}

The considerations of Section \ref{S:prelim} lead to the following
characterization of the range of observation operator $\mathcal{A}$ in terms
of solution of the exterior problem:

\begin{theorem}
\label{T:range implicit} Suppose that function $g$ is defined on $Z_{1}$,
$v^{(g)}$ is the solution to the exterior problem (\ref{E:exterior_problem}),
and $F(\omega,p)$ is defined by (\ref{E:Rvg}). Then function $g$ belongs to
the range of the operator $\mathcal{A}$ (\ref{E:rangeA}) (i.e.,$g=\mathcal{A}%
f$ for some $f\in C_{0}^{\infty}(B)$) \textbf{if and only if} the following
four conditions are satisfied:

\begin{enumerate}
\item Smoothness of $g$: $g\in C_{0}^{\infty}(Z_1)$.

\item Symmetry of $F$: $F(\omega,p)=F(-\omega,-p)$,

\item Smoothness of $F$: $F(\omega,p)\in\mathbb{C}_{0}^{\infty}(S\times
(-1,1))$,

\item Moment conditions for $F$: ${\int\limits_{S}\left[  \int\limits_{-1}%
^{1}F(\omega,p)p^{n}dp\right]  Y_{l}^{\mathbf{m}}(\omega)d\omega
}=0,\mathbf{\ \ \ \ \ \ \ \ \ \ \ \ \ \ \ \ \ \ \ \ \ \ \ \ \ }$%
\ \ \ \ \ \ \ \ \ \ \ \ \ \ \ \ \ \ \ \ \ \ \ \ \ \ \ \ \ \ \ \ for
$n=0,1,2,...,$ $l>n,\mathbf{\ }\left\vert \mathbf{m}\right\vert \leq l.$
\end{enumerate}
\end{theorem}

\begin{proof}
\textbf{Necessity of conditions (1)-(4)}. Smoothness and support of $g$
(condition (1)) follow from the well known properties of solutions of the
initial value problem for the wave equation in $\mathbb{R}^{d}$ with initial
conditions in $C_{0}^{\infty}(\mathbb{R}^{d})$. On the other hand, function
$F(\omega,p)$ defined by (\ref{E:Rvg}) is the Radon transform of $f(x)\in
C_{0}^{\infty}(B)$ and, therefore, satisfies conditions (1)-(3) of Theorem
\ref{T:radon-range}, thus proving statements(2)-(4).

\textbf{Sufficiency of conditions (1)-(4)}. If conditions (2)-(4) are
satisfied, then according to Theorem \ref{T:radon-range} there exists function
$f\in C_{0}^{\infty}(B)$ such that $F(\omega,p)=\mathcal{R}f$. \ One can then
solve the Cauchy problem (\ref{E:originalwave}), (\ref{E:originalBC}) thus
obtaining solution $u(t,x)$ in $(0,1]\times\mathbb{R}^{d}$. The Radon
transform $\mathcal{R}u$ of such a solution is given by formula
(\ref{E:Radon_of_wave}), for each fixed $t\in(0,1]$ and each $\omega\in S.$
This implies that $[\mathcal{R}^{E}u](t,\omega,p)=[\mathcal{R}^{E}%
v^{(g)}](t,\omega,p)$ for all $t\in(0,1],$ $\omega\in S$, and $p\in
\lbrack1,2].$ Since $\mathcal{R}^{E}$ is uniquely invertible, $u(t,x)=v^{(g)}%
(t,x)$ for $t\in(0,1],$ $x\in B^{c}$, and, in particular, $u(t,x)=v^{(g)}%
(t,x)=g(t,x)$ for all $x\in S,$ $t=(0,1],$ implying $g=\mathcal{A}f$.
\end{proof}

\begin{remark}\indent
\begin{enumerate}
\item When considering sufficiency of condition (3) in the above Theorem, one
can notice that $F(\omega,p)$ defined by (\ref{E:Rvg}) is automatically a
$C_{0}^{\infty}$-function on $(-1,0]\times S$ and on $[0,1)\times S$, due to
the smoothness of $v^{(g)}$. Only the continuity of all the derivatives of
$F(\omega,p)$ at $p=0$ needs to be verified. This condition is non-trivial,
since function $F$ is constructed of two "halves" (see equation (\ref{E:Rvg}%
)). This condition is essential, since without it one cannot guarantee
smoothness of the candidate function $f(x)$ at $x=0$.

\item This range description is implicit, since it requires
solving the auxiliary exterior IBVP with the given data $g(x,t)$. The
remaining part of this paper is dedicated to expressing the range conditions
in terms of the Fourier coefficients of the candidate function $b(t,\theta)$.
\end{enumerate}
\end{remark}



\section{A more explicit range characterization}\label{S:explicit}

As mentioned above, Theorem \ref{T:range implicit} provides a
rather implicit range description, which requires, first, solving the exterior
IBVP and then finding its Radon projections. The spherical symmetry of our
domain $\mathbb{R}^{d}\setminus\overline{B}$ will allow us to use standard
separation of variables technique to solve the exterior problem and obtain the
desired range conditions in terms of Fourier coefficients of $b(t,\theta)$.
This requires some preliminary work.

\subsection{Range conditions for the Radon transform in terms of expansion in
spherical harmonics}

In order to avoid the redundancy in the Radon transform defined on
$S\times\lbrack-1,1]$, we will work with the Radon transform restricted to
$S\times\lbrack0,1]$. Let us first formulate the necessary and sufficient
conditions for a function $F(\omega,p)$ defined on $S\times\lbrack0,1]$ to be
in the range of such a restricted Radon transform, i.e.
\begin{equation}
F(\omega,p)=\left[  \mathcal{R}q\right]  (\omega,p),\text{ }\omega\in S,\text{
}p\in\lbrack0,1], \label{E:half_range}%
\end{equation}
for some $q\in C_{0}^{\infty}(B).$

Let us start with necessary conditions. Assume $F=\mathcal{R}q$. An obvious
necessary condition following from Theorem \ref{T:radon-range} is
\begin{equation}
F(\omega,p)\in\mathbb{C}_{0}^{\infty}(S\times\lbrack0,1)). \label{E:half_c0}%
\end{equation}
We need to express remaining conditions in terms of the Fourier coefficients
$F_{l}^{\mathbf{m}}$ of $F.$ Define
\begin{equation}
F_{l}^{\mathbf{m}}(p):=\int\limits_{S}F(\omega,p)Y_{l}^{\mathbf{m}}%
(\omega)d\omega=0,\mathbf{\quad}p\in\lbrack0,1],\mathbf{\quad}%
l=0,1,2,...,\mathbf{\quad}\left\vert \mathbf{m}\right\vert \leq l,
\label{E:half_harmonics}%
\end{equation}
so that
\[
F(\omega,p)=\sum_{l,\mathbf{m}}F_{l}^{\mathbf{m}}(p)Y_{l}^{\mathbf{m}}%
(\omega),\mathbf{\quad}p\in\lbrack0,1].
\]
Note that for any $\omega\in S$,
\begin{equation}[\mathcal{R}q](\omega,-p)=[\mathcal{R}%
q](-\omega,p)=F(\omega,-p) \mbox{ for }p\in\lbrack0,1].
\end{equation}
On the other hand,
\[
F(\omega,-p)=\sum_{l,\mathbf{m}}F_{l}^{\mathbf{m}}(p)Y_{l}^{\mathbf{m}%
}(-\omega)=\sum_{l,\mathbf{m}}(-1)^{l}F_{l}^{\mathbf{m}}(p)Y_{l}^{\mathbf{m}%
}(\omega),\mathbf{\quad}p\in\lbrack0,1].
\]
For any fixed $\omega\in S$, the Radon transform $[\mathcal{R}q](\omega,p)$ is
an infinitely differentiable function of $p$ on the interval $[-1,1]$.
Therefore, functions $^{\ast}F_{l}^{\mathbf{m}}(p)$ defined as%
\begin{equation}
^{\ast}F_{l}^{\mathbf{m}}(p):=\left\{
\begin{array}
[c]{rl}%
F_{l}^{\mathbf{m}}(p), & p\in\lbrack0,1]\\
(-1)^{l}F_{l}^{\mathbf{m}}(-p), & p\in\lbrack-1,0)
\end{array}
\right.  \label{E:two_halves}%
\end{equation}
must be infinitely differentiable on $[-1,1]$, and, in particular, at $p=0.$

Note that $^{\ast}F_{l}^{\mathbf{m}}(p)$ are odd functions for odd $l$, and
even functions for even $l=0,2,4,...$ . Smoothness of these functions implies
that for even $l$, all odd order right derivatives of $F_{l}^{\mathbf{m}}(p)$ at $p=0$
should vanish. For odd $l$, right derivatives of even orders
(including zero) of $F_{l}^{\mathbf{m}}(p)$ must vanish at $p=0$:%
\begin{align}
\frac{\partial^{n}}{\partial p^{n}}F_{l}^{\mathbf{m}}(0)  &  =0,\text{
}l\text{ is even,\ }\left\vert \mathbf{m}\right\vert \leq l,\text{
}n=1,3,5,...,\label{E:0even_der}\\
\frac{\partial^{n}}{\partial p^{n}}F_{l}^{\mathbf{m}}(0)  &  =0,\text{
}l\text{ is odd, \ }\left\vert \mathbf{m}\right\vert \leq l,\text{
}n=0,2,4,...\text{ }. \label{E:0odd_der}%
\end{align}
The moment conditions (\ref{E:standard_moment}) can be expressed in terms of
coefficients $^{\ast}F_{l}^{\mathbf{m}}$ as follows:

\begin{equation}
\int\limits_{-1}^{1}~^{\ast}F_{l}^{\mathbf{m}}(p)p^{n}%
~dp=0,~n=0,1,2,...,,\text{ if }l>n,\text{ even }\left\vert \mathbf{m}\right\vert
\leq l, \label{E:half_moment}%
\end{equation}
or, in terms of $F_{l}^{\mathbf{m}}$
\begin{equation}
\int\limits_{0}^{1}F_{l}^{\mathbf{m}}(p)p^{n}~dp=0,~n=0,1,2,...,,\text{ if
}l>n,\text{ }l+n\text{ even, }\left\vert \mathbf{m}\right\vert \leq l.
\label{E:half_momentnew}%
\end{equation}
Conditions (\ref{E:half_c0}), (\ref{E:0even_der}), (\ref{E:0odd_der}) and
(\ref{E:half_momentnew}) are necessary for a function $F(\omega,p)$ defined on
$S\times\lbrack0,1]$ to be in the range of the Radon transform restricted to
$S\times\lbrack0,1]$. These conditions are also sufficient, as proven in the
theorem below.

\begin{theorem}
\label{P:Radon} Suppose function $F(\omega,p)$ defined on $S\times\lbrack0,1]
$ satisfies conditions (\ref{E:half_c0}), (\ref{E:0even_der}),
(\ref{E:0odd_der}) and (\ref{E:half_momentnew}) with coefficients
$F_{l}^{\mathbf{m}}(p)$ defined by equation (\ref{E:half_harmonics}). Then,
there exists a function $q\in C_{0}^{\infty}(B)$ such that representation
(\ref{E:half_range}) holds.
\end{theorem}

\begin{proof}
Let us define function $^{\ast}F(\omega,p)$ as follows
\[
^{\ast}F(\omega,p)=\left\{
\begin{array}
[c]{rl}%
F(\omega,p), & p\in\lbrack0,1]\\
F(-\omega,-p), & p\in\lbrack-1,0)
\end{array}
\right.  ,\text{ }\omega\in S.
\]
The Fourier coefficients of $^{\ast}F(\omega,p)$ are given by equation
(\ref{E:two_halves}). Therefore, $^{\ast}F(\omega,p)$ satisfies the moment
conditions (\ref{E:standard_moment}) and the symmetry condition $^{\ast
}F(\omega,p)=F(-\omega,-p).$ Moreover, $^{\ast}F(\omega,p)$ is a $C^{\infty}$
function on each of the two halves $S\times(0,1]$ and $S\times\lbrack-1,0)$ of
the cylinder $S\times\lbrack-1,1]$, i.e.
\begin{equation}
^{\ast}F(\omega,p)\in\mathbb{C}_{0}^{\infty}(S\times\lbrack0,1))
\mbox{ and }^{\ast}F(\omega,p)\in\mathbb{C}_{0}^{\infty}(S\times(-1,0]).
\end{equation}
Due to conditions (\ref{E:0even_der}) and (\ref{E:0odd_der}), each of the
coefficients $^{\ast}F_{l}^{\mathbf{m}}(p)$ is a $C_{0}^{\infty}((-1,1))$
function. Thus, for a fixed $\omega$, $^{\ast}F(\omega,p)$ is $C_{0}^{\infty
}((-1,1))$ function in $p.$ Moreover, since all the mixed derivatives of
$F(\omega,p)$ are well defined at $p=0$ for each value of $\omega,$ the limits

\[
\lim_{p\rightarrow0^{\pm}}\,^{\ast}D^{k}F(\omega,p)
\]
of any derivative $D^{k}F$ are well defined and coincide with $D^{k}F(\omega,0).$ It follows that
$^{\ast}F(\omega,p)$ is $C_{0}^{\infty}(S\times(-1,1)),$ and, due to Theorem
\ref{T:radon-range}, $^{\ast}F(\omega,p)$ is in the range of the Radon
transform. In other words, there exists function $q\in C_{0}^{\infty}(B)$ such
that $^{\ast}F(\omega,p)=[\mathcal{R}q](\omega,p).$ Therefore, $F(\omega
,p)=[\mathcal{R}q](\omega,p)$ for $\omega\in S$ and $p\in\lbrack0,1].$
\end{proof}


\subsection{Solving the exterior problem using separation of variables}

\subsubsection{Fourier transform of $b(t,\theta)$ in time}

\phantom{a} \smallskip

\label{SS:FT}

According to the well known Borel extension theorem (see, e.g.
\cite{Wasow}), function $b(t,\theta)$ can be extended (non-uniquely) to $b\in
C_{0}^{\infty}(0,T)$ for any $T>1$. Let us fix such an extension. (Even
though such an extension is arbitrary, we will prove later on that the values
of the extension for $t>1$ do not enter the final formulas).

Further, we
extend $b(t,\theta)$ by zero to all of $\mathbb{R}$ in $t$. Abusing the
notation, we will still call the extended function $b(t,\theta)$.

Such an extension allows us to define the Fourier transform of $b$ with
respect to the time variable $t$:%
\begin{align*}
\hat{b}(\lambda,\theta)  &  :=[\mathcal{F}b](\lambda,\theta):=\int%
\limits_{\mathbb{R}}b(t,\theta)e^{i\lambda t}dt,\\
b(t,\theta)  &  =[\mathcal{F}^{-1}\hat{b}](t,\theta):=\frac{1}{2\pi}%
\int\limits_{\mathbb{R}}\hat{b}(\lambda,\theta)e^{-i\lambda t}d\lambda.
\end{align*}

We proceed by expanding $b(t,\theta)$ and $\hat{b}(\lambda,\theta)$ into
spherical harmonics:
\begin{align}
b(t,\theta)  &  =\sum_{l,\mathbf{m}}b_{l}^{\mathbf{m}}(t)Y_{l}^{\mathbf{m}%
}(\theta),\qquad b_{l}^{\mathbf{m}}(t)=\int\limits_{S}b(t,\theta
)Y_{l}^{\mathbf{m}}(\theta)d\theta,\label{E:b_harm}\\
\hat{b}(\lambda,\theta)  &  =\sum_{l,\mathbf{m}}\hat{b}_{l}^{\mathbf{m}%
}(\lambda)Y_{l}^{\mathbf{m}}(\theta),\qquad\hat{b}_{l}^{\mathbf{m}}%
(\lambda)=\int\limits_{S}\hat{b}(\lambda,\theta)Y_{l}^{\mathbf{m}}%
(\theta)d\theta. \label{E:b_hat_harm}%
\end{align}


\subsubsection{Solution obtained by expansion in spherical harmonics}
\label{S:range}

Our goal is to find the conditions for a candidate function $b(t,\theta)$ to
be in the range of the wave operator $\mathcal{A}$ in terms of the spherical
harmonic expansion coefficients of $b$, rather than in terms of $v^{b}(t,x)$.
To do so, we need to be able to express the exterior Radon transform of the
solution $v^{(b)}$ to the exterior problem (\ref{E:exterior_problem}), in
terms of coefficients $b_{l}^{\mathbf{m}}(t)$ of the spherical harmonic
expansion of $b(t,\theta)$ (see equation (\ref{E:b_harm})).

This will require the use of certain special functions. The \textbf{spherical
Hankel function} $h_{l}^{d}(\lambda)$ is defined through the standard
\textbf{Hankel function} $H_{l+d/2-1}^{(1)}(\lambda)$ of order $l+d/2-1$ by
the following formulas:
\begin{align*}
h_{l}^{d}(\lambda) =\frac{H_{l+d/2-1}^{(1)}(\lambda)}{\lambda^{d/2-1}},
l=0,1,2,..., \lambda\geq0,\\
\\
h_{l}^{d}(-\lambda) =\overline{h_{l}^{d}(\lambda)}, l=0,1,2,...,\lambda>0
\qquad
\end{align*}
(see \cite{OlverNIST10} for details.)

The desired connection between $\mathcal{R}^{E}v^{(b)}$ and $b_{l}%
^{\mathbf{m}}$ is given by the following theorem.

\begin{theorem}
\label{T:series}Let $v^{(b)}$ be the solution to the exterior problem
(\ref{E:exterior_problem}) with boundary condition $b(t,\theta)\in C_{0}%
(Z_{1})$, and $[\mathcal{R}^{E}v^{(b)}](t,\omega,p)$ be the exterior Radon
transform of $v^{(b)}$. Then $\mathcal{R}^{E}v^{(b)}$ can be expressed as the
following series%
\begin{equation}
\lbrack\mathcal{R}^{E}v^{(b)}](t,\omega,p)=\sum_{l,\mathbf{m}}R_{l}%
^{\mathbf{m}}(t-p)Y_{l}^{\mathbf{m}}(\omega), \label{E:Rml}%
\end{equation}
with functions $R_{l}^{\mathbf{m}}$ expressed through the Fourier coefficients
$\hat{b}_{l}^{\mathbf{m}}(\lambda)$ (see equation (\ref{E:b_hat_harm})) as
follows:
\begin{equation}
R_{l}^{\mathbf{m}}(t)=-2^{^{\frac{d}{2}}}\pi^{^{\frac{d}{2}-1}}i^{l}\left[
\mathcal{F}^{-1}\left(  \frac{\hat{b}_{l}^{\mathbf{m}}(\lambda)}{\lambda
^{d-1}h_{l}^{d}(\lambda)}\right)  \right]  (t). \label{E:Fourier_Rml}%
\end{equation}
Alternatively, functions $R_{l}^{\mathbf{m}}$ can be represented as
convolutions
\begin{equation}
R_{l}^{\mathbf{m}}(t)=(b_{l}^{\mathbf{m}}\ast K_{l})(t)=\int\limits_{0}%
^{1+t}b_{l}^{\mathbf{m}}(s)K_{l}(t-s)ds, \label{E:varlimits}%
\end{equation}
with convolution kernels $K_{l}(t)$ defined through the inverse Fourier transform as
follows%
\begin{equation}
K_{l}(t)=-\left[  \mathcal{F}^{-1}\left(  \frac{2^{^{\frac{d}{2}}}\pi
^{^{\frac{d}{2}-1}}i^{l}}{\lambda^{d-1}h_{l}^{d}(\lambda)}\right)  \right]
(t). \label{E:convolution_kernels}%
\end{equation}

\end{theorem}
We postpone the lengthy proof of this Theorem till Section \ref{S:proofs}.

\begin{remark}The above theorem, by solving explicitly the exterior problem,
relates the Radon transform of $v^{(b)}$ directly to the Fourier coefficients
of $b$. Below we use these formulas at the value of ${t=1}$ and values of $p$
lying in the interval $[1,2]$. Then, integration in the
equation~(\ref{E:varlimits}) requires only the values of $b_{l}^{\mathbf{m}%
}(s)$ in the interval $s\in\lbrack0,1]$. Thus, our choice of the arbitrary
smooth extension of $b(t,\theta$) to times $t>1$ does not affect functions
$R_{l}^{\mathbf{m}}(t)$ that are being computed.
\end{remark}


\section{Range theorem in terms of spherical harmonics expansion}\label{S:range_expansions}


\begin{theorem}
\label{T:main} A function $b(t,\theta)\in C_{0}^{\infty}(Z_{1})$ is in the
range of $\mathcal{A}$, i.e. can be expressed as $b=\mathcal{A}f$ if and only
if the following two sets of conditions are satisfied:

\begin{enumerate}
\item \textbf{moment conditions}%
\[
\int\limits_{0}^{1}R_{l}^{\mathbf{m}}(p)p^{n}~dp=0,~n=0,1,2,...,,\text{ if
}l>n,\text{ l+n is even, }\left\vert \mathbf{m}\right\vert \leq l.
\]

\item \textbf{``smoothness'' conditions} at $p=0$%
\begin{align*}
\frac{\partial^{n}}{\partial p^{n}}R_{l}^{\mathbf{m}}(0)  &  =0,\text{
}l\text{ is even,\ }\left\vert \mathbf{m}\right\vert \leq l,\text{
}n=1,3,5,...,\\
\frac{\partial^{n}}{\partial p^{n}}R_{l}^{\mathbf{m}}(0)  &  =0,\text{
}l\text{ is odd, \ }\left\vert \mathbf{m}\right\vert \leq l,\text{
}n=0,2,4,...\text{ }.
\end{align*}
where functions $R_{l}^{\mathbf{m}}(t)$, $l=0,1,2,...,$ $|\mathbf{m}|\leq l,$
are defined by (\ref{E:Rml}).
\end{enumerate}

Moreover, if conditions (1) and (2) are satisfied, the Radon projections
$F(\omega,p)$ of $f(x)$ are given explicitly by the formulas%
\[
F(\omega,p)=\sum_{l,\mathbf{m}}R_{l}^{\mathbf{m}}(p)Y_{l}^{\mathbf{m}}%
(\omega),\ F(\omega,-p)=F(-\omega,p),\ p\in\lbrack0,1],\ \omega\in S,
\]
and so $f$ can be reconstructed from $F$ by inverting the Radon transform.
\end{theorem}

\begin{proof}
Define function $F(\omega,p)$ for $p\in\lbrack0,1]$,\ $\omega\in S$, as
follows
\[
F(\omega,p)=[\mathcal{R}^{E}v^{(b)}](1,\omega,p+1)=\sum_{l,\mathbf{m}}%
R_{l}^{\mathbf{m}}(p)Y_{l}^{\mathbf{m}}(\omega).
\]
The coefficients of the spherical harmonic expansion of $F(\omega,p)$ are%
\[
F_{l}^{\mathbf{m}}(p)=\int\limits_{S}F(\omega,p)Y_{l}^{\mathbf{m}}%
(\omega)d\omega=R_{l}^{\mathbf{m}}(p),
\]
with functions $R_{l}^{\mathbf{m}}$ defined by equations (\ref{E:varlimits})
and (\ref{E:convolution_kernels}). Proposition~\ref{P:Radon} yields the
necessary and sufficient conditions on $F(\omega,p)$ and $F_{l}^{\mathbf{m}%
}(p)$ for $b$ being in the range of $\mathcal{A}$. These conditions coincide
with the moment and smoothness conditions of the theorem. This completes the proof.
\end{proof}


\section{Proof of Theorem \ref{T:series}}
\label{S:proofs}

We need to prepare several preliminary results before Theorem \ref{T:series}
can be proven.

First, let us use separation of variables to obtain explicit expression for
the solution $v^{(b)}(t,x)$ to the exterior problem (\ref{E:exterior_problem}%
). We will also need the time derivative $v_{t}^{(b)}(t,x)$ of $v^{(b)}(t,x),
$ and the expression for the normal derivative $\frac{\partial}{\partial
n}v^{(b)}(t,\theta)$, where $\theta\in S.$

\begin{proposition}
For a given condition $b(t,\theta),$ solution $v^{(b)}(t,x)$ to the exterior
problem (\ref{E:exterior_problem}) and its derivatives are given by the
following formulas:%
\begin{align}
v^{(b)}(t,x)  &  =\frac{1}{2\pi}\int\limits_{\mathbb{R}}\left(  \sum
_{l,\mathbf{m}}\hat{b}_{l}^{\mathbf{m}}(\lambda)\frac{h_{l}^{d}(\lambda
|x|)}{h_{l}^{d}(\lambda)}Y_{l}^{\mathbf{m}}(\hat{x})\right)  e^{-i\lambda
t}d\lambda,\label{E:ext_sol}\\
v_{t}^{(b)}(t,x)  &  =-\frac{1}{2\pi}\int\limits_{\mathbb{R}}\left(
i\lambda\sum_{l,\mathbf{m}}\hat{b}_{l}^{\mathbf{m}}(\lambda)\frac{h_{l}%
^{d}(\lambda|x|)}{h_{l}^{d}(\lambda)}Y_{l}^{\mathbf{m}}(\hat{x})\right)
e^{-i\lambda t}d\lambda,\label{E:ext_tder}\\
\frac{\partial}{\partial n}v^{(b)}(t,\theta)  &  =\frac{1}{2\pi}%
\int\limits_{\mathbb{R}}\left(  \lambda\sum_{l,\mathbf{m}}\hat{b}%
_{l}^{\mathbf{m}}(\lambda)\frac{(h_{l}^{d}(\lambda))^{\prime}}{h_{l}%
^{d}(\lambda)}Y_{l}^{\mathbf{m}}(\theta)\right)  e^{-i\lambda t}d\lambda.
\label{E:ext_nder}%
\end{align}

\end{proposition}

\begin{proof}
We observe that the following combinations of functions%
\[
\frac{h_{l}^{d}(\lambda|x|)}{h_{l}^{d}(\lambda)}e^{-i\lambda t}Y_{l}%
^{\mathbf{m}}(\hat{x}),\qquad\hat{x}=x/|x|,
\]
are outgoing solutions of Helmholz equation.
With $x$ restricted to  $S$, these functions are equal to $e^{-i\lambda
t}Y_{l}^{\mathbf{m}}(x/|x|)$. Thus, they form a complete orthonormal system on $\mathbb{R}\times S$.

Therefore, the solution $v^{(b)}$ for the
exterior problem corresponding to the boundary values $b(t,\theta)$ can be
written in the form (\ref{E:ext_sol}). Due to the smoothness of the function
$b(t,\theta)$ the series converges fast (faster than any inverse power of
$|\mathbf{m}|$). This justifies the term-wise differentiation of the series,
resulting in the formulas (\ref{E:ext_tder}) and (\ref{E:ext_nder}).
\end{proof}

Our intention is to compute the exterior Radon transform of $v^{(b)}(t,x)$
given by formula (\ref{E:ext_sol}). A natural idea would be to try to
integrate equation (\ref{E:ext_sol}) at each frequency $\lambda$ term-by-term.
However, at dimensions $d$ higher than $2,$ integrals over planes of functions
$h_{l}^{d}(\lambda|x|)Y_{l}^{\mathbf{m}}(\hat{x})$ diverge, due to the slow
decrease of $h_{l}^{d}(\lambda|x|)$ at infinity. Thus, we need the following
workaround that expresses integrals of $v_{t}^{(b)}(t,x)$ over planes through
integrals over sphere $S.$ Since the sphere is bounded, the question of
convergence does not arise.

\begin{proposition}
The exterior Radon transform $\mathcal{R}^{E}v_{t}^{(b)}$ of the time
derivative $v_{t}^{(b)}$ of the exterior solution $v^{(b)}(t,x)$ can be
expressed through integrals over sphere $S$ as follows:
\begin{align}
&  2[\mathcal{R}^{E}v_{t}^{(b)}](t,\omega,p)=\label{E:sphere_integrals}\\
&  \int\limits_{S}\frac{\partial}{\partial n}v^{(b)}(\omega\cdot
\theta-p+t,\theta)d\theta-\int\limits_{S}v_{t}^{(b)}(\omega\cdot
\theta-p+t,\theta)(\theta\cdot\omega)d\theta.\nonumber
\end{align}

\end{proposition}

\begin{proof}
The well known Kirchhoff representation (see, e.g., \cite{laliena} and
references therein) allows us to express $v^{(b)}$, using Green's function
$\mathbf{G}(t,x)$, through the boundary values of $v^{(b)}$ and $\frac
{\partial}{\partial n}v^{(b)}$ as follows;%
\begin{align*}
&  v^{(b)}(t,x)\\
&  =\int\limits_{S}\int\limits_{0}^{t}\left(  \frac{\partial}{\partial
n}v^{(b)}(\tau,\theta)\mathbf{G}(t-\tau,x-\theta)-v^{(b)}(\tau,\theta
)\frac{\partial}{\partial n(\theta)}\mathbf{G}(t-\tau,x-\theta)\right)  d\tau
d\theta.
\end{align*}
Correspondingly, the time derivative $v_{t}^{(b)}(t,x)$ can be represented by
the following formula:%
\begin{align*}
&  v_{t}^{(b)}(t,x)\\
&  =\int\limits_{S}\int\limits_{0}^{t}\left(  \frac{\partial}{\partial
n}v^{(b)}(\tau,\theta)\mathbf{G}_{t}(t-\tau,x-\theta)-v^{(b)}(\tau
,\theta)\frac{\partial}{\partial n(\theta)}\mathbf{G}_{t}(t-\tau
,x-\theta)\right)  d\tau d\theta,
\end{align*}
where $\mathbf{G}_{t}$ denotes the time derivative of $\mathbf{G}$. Let us
evaluate the exterior Radon transform $[\mathcal{R}^{E}v_{t}^{(b)}%
](t,\omega,p)$ of $v_{t}^{(b)}$, for $p\geq1$. For each fixed $t$ we obtain%
\begin{align*}
&  [\mathcal{R}^{E}v_{t}^{(b)}](t,\omega,p)\\
&  =\int\limits_{\Pi(\omega,p)}\left[  \int\limits_{S}\int\limits_{0}%
^{t}\left(  \frac{\partial}{\partial n}v^{(b)}(\tau,\theta)\mathbf{G}%
_{t}(t-\tau,x-\theta)\right.  \right. \\
&  \phantom{aaaaaaaa} \left.  \left.  -v^{(b)}\frac{\partial}{\partial
n(\theta)}\mathbf{G}_{t}(t-\tau,x-\theta)\right)  d\tau d\theta\right]  dx\\
&  =\int\limits_{S}\int\limits_{0}^{t}\frac{\partial}{\partial n}v^{(b)}%
(\tau,\theta)\left[  \int\limits_{\Pi(\omega,p)}\mathbf{G}_{t}(t-\tau
,x-\theta)dx\right]  d\tau d\theta\\
&  -\int\limits_{S}\int\limits_{0}^{t}v^{(b)}(\tau,\theta)\left[  \theta
\cdot\nabla_{\theta}\int\limits_{\Pi(\omega,p)}\mathbf{G}_{t}(t-\tau
,x-\theta)dx\right]  d\tau d\theta,
\end{align*}
where reversing the order of integration is justified since for any fixed $t$
function $v^{(b)}(t,x)$ is finitely supported, and where we equate the normal
to the unit sphere $S$ at the point $\theta,$ with $\theta.$ We now observe
that, for $t>0$,%
\begin{align*}
2\int\limits_{\Pi(\omega,p)}\mathbf{G}_{t}(t,x-y)dx  &  =\delta(\omega\cdot
y-p-t)+\delta(\omega\cdot y-p+t),\\
2\nabla_{y}\int\limits_{\Pi(\omega,p)}\mathbf{G}_{t}(t,x-y)dx  &
=\omega\left[  \delta^{\prime}(\omega\cdot y-p-t)+\delta^{\prime}(\omega\cdot
y-p+t)\right]  .
\end{align*}
Each of the above equations represents two singular waves propagating in the
opposite directions $\pm\omega$ as $t$ increases. Now with a substitution
$t\rightarrow t-\tau$, for $t>\tau$ we obtain%
\begin{align*}
2\int\limits_{\Pi(\omega,p)}\mathbf{G}_{t}(t-\tau,x-y)dx  &  =\delta
(\omega\cdot y-p-(t-\tau))+\delta(\omega\cdot y-p+t-\tau),\\
2\nabla_{y}\int\limits_{\Pi(\omega,p)}\mathbf{G}_{t}(t-\tau,x-y)dx  &
=\omega\left[  \delta^{\prime}(\omega\cdot y-p-(t-\tau))+\delta^{\prime
}(\omega\cdot y-p+t-\tau)\right]  .
\end{align*}
Since we consider values of $p\geq1$ and $\tau\leq t,$ only one of the
terms in the above formulas will contribute to the integral over sphere~$S$:%
\begin{align}
&  2[\mathcal{R}^{E}v_{t}^{(b)}](t,\omega,p)\nonumber\\
&  =\int\limits_{S}\int\limits_{0}^{t}\frac{\partial}{\partial n}v^{(b)}%
(\tau,\theta)\delta(\omega\cdot\theta-p+t-\tau)d\tau d\theta\nonumber\\
&  -\int\limits_{S}\int\limits_{0}^{t}v^{(b)}(t,\theta)(\theta\cdot
\omega)\delta^{\prime}(\omega\cdot\theta-p+t-\tau)d\tau d\theta\nonumber\\
&  =\int\limits_{S}\frac{\partial}{\partial n}v^{(b)}(\omega\cdot
\theta-p+t,\theta)d\theta-\int\limits_{S}v_{t}^{(b)}(\omega\cdot
\theta-p+t,\theta)(\theta\cdot\omega)d\theta. \label{E:trace_integrals}%
\end{align}

\end{proof}

Next, we formulate and prove the following proposition.

\begin{proposition}
\label{T:just_series}Let $v^{(b)}$ be the solution to the exterior problem
(\ref{E:exterior_problem}) with boundary condition $b(t,\theta)\in C_{0}%
(Z_{1})$, and $[\mathcal{R}^{E}v^{(b)}](t,\omega,p)$ be the exterior Radon
transform of $v^{(b)}$. Then $\mathcal{R}^{E}v^{(b)}$ can be expressed as the
series (\ref{E:Rml}), with functions $R_{l}^{\mathbf{m}}$ expressed through
the Fourier coefficients $\hat{b}_{l}^{\mathbf{m}}(\lambda)$ as shown in
equation (\ref{E:Fourier_Rml}). Moreover, functions $R_{l}^{\mathbf{m}}(t)$
can be represented as convolutions
\begin{equation}
R_{l}^{\mathbf{m}}(t)=(b_{l}^{\mathbf{m}}\ast K_{l})(t)=\int%
\limits_{\mathbb{R}}b_{l}^{\mathbf{m}}(s)K_{l}(t-s)ds,
\label{E:preliminary_conv}%
\end{equation}
with kernels $K_{l}(t)$ defined by equation (\ref{E:convolution_kernels});
these functions are infinitely differentiable for all values of $t.$
\end{proposition}

\begin{proof}
Let us utilize the expression for $\mathcal{R}^{E}v_{t}^{(b)}$ given by
equation (\ref{E:trace_integrals}), and substitute into it expressions
(\ref{E:ext_tder}) and (\ref{E:ext_nder}) for $v_{t}^{(b)}$ and $\frac
{\partial}{\partial n}v^{(b)}$. Then the second integral on the last line of
(\ref{E:trace_integrals}) can be transformed as follows:
\begin{align}
&  -\int\limits_{S}v_{t}^{(b)}(\omega\cdot y-p+t,\theta)(\theta\cdot
\omega)d\theta\nonumber\\
&  =\frac{1}{2\pi}\int\limits_{S}\left[  \int\limits_{\mathbb{R}}\left\{
\sum_{l,\mathbf{m}}\hat{b}_{l}^{\mathbf{m}}(\lambda)Y_{l}^{\mathbf{m}}%
(\theta)\right\}  i\lambda e^{-i\lambda\omega\cdot\theta}(\theta\cdot
\omega)e^{-i\lambda(t-p)}d\lambda\right]  d\theta\nonumber\\
&  =-\frac{1}{2\pi}\int\limits_{S}\left[  \int\limits_{\mathbb{R}}\left\{
\sum_{l,\mathbf{m}}\hat{b}_{l}^{\mathbf{m}}(\lambda)Y_{l}^{\mathbf{m}}%
(\theta)\right\}  \lambda\left\{  \frac{d}{d\lambda}e^{-i\lambda\omega
\cdot\theta}\right\}  e^{-i\lambda(t-p)}d\lambda\right]  d\theta\nonumber\\
&  =-\frac{1}{2\pi}\int\limits_{\mathbb{R}}\lambda e^{-i\lambda(t-p)}\left[
\sum_{l,\mathbf{m}}\hat{b}_{l}^{\mathbf{m}}(\lambda)\frac{d}{d\lambda}\left\{
\int\limits_{S}Y_{l}^{\mathbf{m}}(\theta)e^{-i\lambda\omega\cdot\theta}%
d\theta\right\}  \right]  d\lambda\label{E:awful_mess_2}%
\end{align}
In order to simplify the expression in the curly braces, we use Jacobi-Anger
expansion formula (Lemma 9.10.2 in \cite{Askey}) which in our notation has the
following form:%
\[
\int\limits_{S}e^{-2\pi it(\hat{x}\cdot\hat{y})}Y_{l}^{\mathbf{m}}(\hat
{x})d\hat{x}=2\pi i^{l}Y_{l}^{\mathbf{m}}(\hat{y})\frac{J_{l+\frac{d}{2}%
-1}(2\pi t)}{t^{\frac{d}{2}-1}},
\]
or, with substitution $2\pi t=\lambda$, $\hat{y}=\omega,$ $\hat{x}=\theta,$
\begin{equation}
\int\limits_{S}e^{-i\lambda\omega\cdot\theta}Y_{l}^{\mathbf{m}}(\theta
)d\theta=(2\pi)^{\frac{d}{2}}i^{l}Y_{l}^{\mathbf{m}}(\omega)j_{l}^{d}%
(\lambda).\label{E:JacobiAnger}%
\end{equation}
where%
\[
j_{l}^{d}(\lambda)=\operatorname{Re}h_{l}^{d}(\lambda)=\frac{J_{l+\frac{d}%
{2}-1}(\lambda)}{\lambda^{\frac{d}{2}-1}},\qquad l=0,1,2,....
\]
Now, using (\ref{E:JacobiAnger}), equation (\ref{E:awful_mess_2}) takes the
following form:%
\begin{align}
&  -\int\limits_{S}v_{t}^{(b)}(\omega\cdot y-p+t,\theta)(\theta\cdot
\omega)d\theta\nonumber\\
&  =-(2\pi)^{^{\frac{d}{2}-1}}\int\limits_{\mathbb{R}}\left[  \sum
_{l,\mathbf{m}}\hat{b}_{l}^{\mathbf{m}}(\lambda)i^{l}Y_{l}^{\mathbf{m}}%
(\omega)(j_{l}^{d}(\lambda))^{\prime}\right]  \lambda e^{-i\lambda
(t-p)}d\lambda.\label{E:second-term}%
\end{align}
\indent A similar technique can be used to simplify the first term
in\ (\ref{E:trace_integrals}). Indeed, taking into account formula
(\ref{E:ext_nder}) we obtain%
\begin{align*}
&  \int\limits_{S}\frac{\partial}{\partial n}v^{(b)}(\omega\cdot
\theta-p+t,\theta)d\theta\\
&  =\frac{1}{2\pi}\int\limits_{S}\left[  \int\limits_{\mathbb{R}}\left\{
\sum_{l,\mathbf{m}}\hat{b}_{l}^{\mathbf{m}}(\lambda)\frac{\lambda(h_{l}%
^{d}(\lambda))^{\prime}}{h_{l}^{d}(\lambda)}Y_{l}^{\mathbf{m}}(\theta
)e^{-i\lambda\omega\cdot\theta}\right\}  e^{-i\lambda(t-p)}d\lambda\right]
d\theta\\
&  =\frac{1}{2\pi}\int\limits_{\mathbb{R}}\lambda\left(  \sum_{l,\mathbf{m}%
}\hat{b}_{l}^{\mathbf{m}}(\lambda)\frac{(h_{l}^{d}(\lambda))^{\prime}}%
{h_{l}^{d}(\lambda)}\left[  \int\limits_{S}e^{-i\lambda\omega\cdot\theta}%
Y_{l}^{\mathbf{m}}(\theta)d\theta\right]  \right)  e^{-i\lambda(t-p)}d\lambda.
\end{align*}
Again, using the Jacobi-Anger formula to compute the term in the brackets,%
\begin{align}
&  \int\limits_{S}\frac{\partial}{\partial n}v^{(b)}(\omega\cdot
\theta-p+t,\theta)d\theta\nonumber\\
&  =(2\pi)^{^{\frac{d}{2}-1}}\int\limits_{\mathbb{R}}\lambda\left(
\sum_{l,\mathbf{m}}i^{l}\hat{b}_{l}^{\mathbf{m}}(\lambda)\frac{\left(
h_{l}^{d}\right)  ^{\prime}(\lambda)}{h_{l}^{d}(\lambda)}Y_{l}^{\mathbf{m}%
}(\omega)j_{l}^{d}(\lambda)\right)  e^{-i\lambda(t-p)}d\lambda
.\label{E:first_term}%
\end{align}
Now we can combine expressions (\ref{E:first_term}) and (\ref{E:second-term})
for the terms appearing in (\ref{E:trace_integrals}):%
\begin{align*}
2[\mathcal{R}^{E}v_{t}^{(b)}](t,\omega,p) &  =(2\pi)^{^{\frac{d}{2}-1}}%
\int\limits_{\mathbb{R}}\lambda\left(  \sum_{l,\mathbf{m}}i^{l}\hat{b}%
_{l}^{\mathbf{m}}(\lambda)\frac{\left(  h_{l}^{d}\right)  ^{\prime}(\lambda
)}{h_{l}^{d}(\lambda)}Y_{l}^{\mathbf{m}}(\omega)j_{l}^{d}(\lambda)\right)
e^{-i\lambda(t-p)}d\lambda\\
&  -(2\pi)^{^{\frac{d}{2}-1}}\int\limits_{\mathbb{R}}\lambda\left(
\sum_{l,\mathbf{m}}i^{l}\hat{b}_{l}^{\mathbf{m}}(\lambda)Y_{l}^{\mathbf{m}%
}(\omega)(j_{l}^{d})^{\prime}(\lambda)\right)  e^{-i\lambda(t-p)}d\lambda
\end{align*}%
\[
=(2\pi)^{^{\frac{d}{2}-1}}\int\limits_{\mathbb{R}}\lambda\left(
\sum_{l,\mathbf{m}}i^{l}\hat{b}_{l}^{\mathbf{m}}(\lambda)\frac{\left(
h_{l}^{d}\right)  ^{\prime}(\lambda)j_{l}^{d}(\lambda)-h_{l}^{d}%
(\lambda)(j_{l}^{d}(\lambda))^{\prime}}{h_{l}^{d}(\lambda)}Y_{l}^{\mathbf{m}%
}(\omega)j_{l}^{d}(\lambda)\right)  e^{-i\lambda(t-p)}d\lambda
\]
The following well-known identity for the Wronskian of functions $J_{\mu}$ and
$H_{\mu}^{(1)}$ (formula 10.5.3 in \cite{OlverNIST10})%
\[
\mathcal{W}(J_{\mu},H_{\mu}^{(1)})(\lambda)=2i/(\pi\lambda)
\]
yields%
\begin{align}\label{E:Wronskian}
\left(  h_{l}^{d}(\lambda)\right)  ^{\prime}j_{l}^{d}(\lambda)-h_{l}%
^{d}(\lambda)(j_{l}^{d}(\lambda))^{\prime}\\
=\frac{1}{\lambda^{d-2}}%
\mathcal{W}\left(  J_{l+d/2-1},H_{l+d/2-1}^{(1)}\right)  (\lambda)=\frac
{2i}{\pi\lambda^{d-1}},%
\end{align}
which, in turn, leads to a simplification of the expression for $\mathcal{R}%
^{E}v_{t}^{(b)}$:%
\begin{align*}
\lbrack\mathcal{R}^{E}v_{t}^{(b)}](t,\omega,p) &  =(2\pi)^{^{\frac{d}{2}-1}%
}\frac{i}{\pi}\int\limits_{\mathbb{R}}\left(  \sum_{l,\mathbf{m}}i^{l}%
\frac{\hat{b}_{l}^{\mathbf{m}}(\lambda)}{\lambda^{d-2}h_{l}^{d}(\lambda)}%
Y_{l}^{\mathbf{m}}(\omega)\right)  e^{-i\lambda(t-p)}d\lambda\\
&  =2^{^{\frac{d}{2}-1}}\pi^{^{\frac{d}{2}-2}}\sum_{l,\mathbf{m}}\left(
i^{l+1}\int\limits_{\mathbb{R}}\frac{\hat{b}_{l}^{\mathbf{m}}(\lambda
)}{\lambda^{d-2}h_{l}^{d}(\lambda)}e^{-i\lambda(t-p)}d\lambda\right)
Y_{l}^{\mathbf{m}}(\omega).
\end{align*}
Since $v^{(b)}(0,x)=0$, we have $[\mathcal{R}^{E}v^{(b)}](0,\omega,p)=0.$ Now
the Radon transform $\mathcal{R}^{E}v^{(b)}$ of $v^{(b)}$ can be found by
integration in time $t$:%
\begin{align*}
\lbrack\mathcal{R}^{E}v^{(b)}](t,\omega,p) &  =\int\limits_{0}^{t}%
[\mathcal{R}^{E}v_{\tau}^{(b)}](\tau,\omega,p)d\tau\\
&  =2^{^{\frac{d}{2}-1}}\pi^{^{\frac{d}{2}-2}}\int\limits_{0}^{t}\left[
\sum_{l,\mathbf{m}}\left(  i^{l+1}\int\limits_{\mathbb{R}}\frac{\hat{b}%
_{l}^{\mathbf{m}}(\lambda)}{\lambda^{d-2}h_{l}^{d}(\lambda)}e^{-i\lambda\tau
}e^{i\lambda p}d\lambda\right)  Y_{l}^{\mathbf{m}}(\omega)\right]  d\tau\\
&  =-2^{^{\frac{d}{2}-1}}\pi^{^{\frac{d}{2}-2}}\sum_{l,\mathbf{m}}\left(
i^{l}\int\limits_{\mathbb{R}}\frac{\hat{b}_{l}^{\mathbf{m}}(\lambda)}%
{\lambda^{d-1}h_{l}^{d}(\lambda)}e^{-i\lambda(t-p)}d\lambda\right)
Y_{l}^{\mathbf{m}}(\omega),
\end{align*}
or%
\begin{align}
\lbrack\mathcal{R}^{E}v^{(b)}](t,\omega,p) &  =\sum_{l,\mathbf{m}}%
R_{l}^{\mathbf{m}}(t-p)Y_{l}^{\mathbf{m}}(\omega)\text{ with }%
\label{E:Radon_of_v}\\
R_{l}^{\mathbf{m}}(t) &  =-2^{^{\frac{d}{2}-1}}\pi^{^{\frac{d}{2}-2}}i^{l}%
\int\limits_{\mathbb{R}}\frac{\hat{b}_{l}^{\mathbf{m}}(\lambda)}{\lambda
^{d-1}h_{l}^{d}(\lambda)}e^{-i\lambda t}d\lambda\nonumber\\
&  =-2^{^{\frac{d}{2}}}\pi^{^{\frac{d}{2}-1}}i^{l}\left[  \mathcal{F}%
^{-1}\left(  \frac{\hat{b}_{l}^{\mathbf{m}}(\lambda)}{\lambda^{d-1}h_{l}%
^{d}(\lambda)}\right)  \right]  (t).\label{E:RLM1}%
\end{align}
Equations (\ref{E:Radon_of_v}) and (\ref{E:RLM1}) coincide with equations
(\ref{E:Rml}) and (\ref{E:Fourier_Rml}) announced in Theorem (\ref{T:series})
and in the statement of this proposition. It follows that functions
$R_{l}^{\mathbf{m}}(t)$ can be represented as convolutions
(\ref{E:preliminary_conv}). Moreover, since we have extended function
$b(t,\theta)$ to a $C_{0}^{\infty}(\mathbb{R}\times S)$ function, coefficients
$b_{l}^{\mathbf{m}}(t)$ are infinitely smooth. Therefore, functions
$R_{l}^{\mathbf{m}}(t)$ are also $C_{0}^{\infty}(\mathbb{R})$ functions, as
convolutions of distributions with $C_{0}^{\infty}$ functions.
\end{proof}
\begin{remark}\indent
\begin{enumerate}
\item The appearance of convolutions in (\ref{E:preliminary_conv}) is not
surprising, since the problem we consider is invariant with respect to the
shift in time.

\item We will only need to know the values of $R_{l}^{\mathbf{m}}(t)$ on the
interval $[-1,0]$, and the values of all the derivatives of $R_{l}%
^{\mathbf{m}}(t)$ at $t=0.$ In what follows we will show that distributions
$K_{l}(t)$ have finite support, and as a result, integration in
(\ref{E:convolution_kernels}) is performed over finite intervals, such that
the values of the smooth extension of $b(t,\theta)$ to $t>1$ are, in fact, not
used at all.
\end{enumerate}
\end{remark}

\begin{proposition}
Convolution kernels $K_{l}(t)$ given by equation (\ref{E:convolution_kernels})
vanish for $t<-1$ in the sense of distributions, and thus integration interval
in (\ref{E:preliminary_conv}) can be reduced as
follows:\label{T:finite_interval}%
\begin{equation}
R_{l}^{\mathbf{m}}(\tau)=\int\limits_{0}^{1+\tau}b_{l}^{\mathbf{m}}%
(s)K_{l}(\tau-s)ds.\label{E:interval}%
\end{equation}
Moreover, in order to compute values of $R_{l}^{\mathbf{m}}(t)$ in the
interval $t\in\lbrack-1,0]$ using the above formula, functions $b_{l}%
^{\mathbf{m}}(t)$ need to be known only on the interval $t\in\lbrack0,1].$
\end{proposition}

\begin{proof}
Consider an arbitrary $C_{0}^{\infty}(\mathbb{R\times}S)$ function
$b(t,\theta)$, finitely supported in $t$ on the interval $(0,T)$ for some
$T>1,$ and function $v^{(b)}(t,x)$ that is a solution to the exterior problem
in $\mathbb{R}\times B^{c}$, vanishing at $t=-\infty$:
\[
\left\{
\begin{array}
[c]{ll}%
\frac{\partial^{2}v^{(b)}(t,x)}{\partial t^{2}}-\Delta v^{(b)}(t,x)=0, &
(t,x)\in\mathbb{R}\times B^{c},\\
v^{(b)}(t,x)=0,\quad v_{t}^{(b)}(t,x)=0, & (t,x)\in(-\infty,0)\times B^{c},\\
v^{(b)}(t,\theta)=b(t,\theta). & (t,\theta)\in\mathbb{R}\times S.
\end{array}
\right.
\]
Due to the finite speed of sound, solution $v^{(b)}(t,x)$ of this problem is
identically zero outside of the interior of the cone $|x|\leq t+1.$ Therefore,
for the exterior Radon transform of $v^{(b)}$ we have%
\[
\lbrack\mathcal{R}^{E}v^{(b)}](t,\omega,p)=0\text{ for }p-t>1,\quad\omega\in
S.
\]
In terms of Fourier coefficients of $\mathcal{R}^{E}$ (see equation
(\ref{E:Radon_of_v})), this implies
\[
R_{l}^{\mathbf{m}}(t-p)=0\text{ for }t-p<-1,\quad l=0,1,2,3,..,\quad\left\vert
\mathbf{m}\right\vert \leq l,
\]
or, equivalently $R_{l}^{\mathbf{m}}(t)=0$ for $t<-1.$ Values of the Fourier
coefficients $R_{l}^{\mathbf{m}}(t)$ can be expressed through the Fourier
coefficients $b_{l}^{\mathbf{m}}(t)$ of \ $b(t,\theta)$ by convolutions
(\ref{E:preliminary_conv}). Since $b(t,\theta)$ vanishes for $t<0$,
coefficients $b_{l}^{\mathbf{m}}(t)$ also vanish for these values of $t$.
Therefore, equation (\ref{E:preliminary_conv}) can be re-written as follows
\begin{equation}
R_{l}^{\mathbf{m}}(t)=\int\limits_{0}^{\infty}b_{l}^{\mathbf{m}}%
(s)K_{l}(t-s)ds.\label{E:simple_convolution}%
\end{equation}
Since function $b(t,\theta)$ is arbitrary, so are coefficients $b_{l}%
^{\mathbf{m}}(t)$ (subject to conditions of being smooth and supported on
$(0,T))$. Then, for a fixed $l,\mathbf{m}$, vanishing of $R_{l}^{\mathbf{m}%
}(t)$ for $t<-1$ implies vanishing of the integral (\ref{E:simple_convolution}%
) for $t<-1,$ for an arbitrary test function $b_{l}^{\mathbf{m}}(t)$. It
follows that for any $t<-1,$ distribution $K_{l}(t-s)$ vanishes for all
$s\geq0.$ In other words,%
\[
K_{l}(\tau)=0\text{ for }\tau<-1,\quad l=0,1,2,3,..,\quad\left\vert
\mathbf{m}\right\vert \leq l,
\]
as we wanted to show. This implies that the integration interval in
(\ref{E:simple_convolution}) can be further reduced, yielding equation
(\ref{E:interval}). Clearly, in order to compute values of $R_{l}^{\mathbf{m}%
}(t)$ in the interval $t\in\lbrack-1,0]$ using formula (\ref{E:interval})
functions $b_{l}^{\mathbf{m}}(t)$ need to be known only on the interval
$t\in\lbrack0,1]$.This completes the proof of the proposition.
\end{proof}

Finally, the proof of Theorem~\ref{T:series} results by combining the
statements of Propositions~\ref{T:just_series} and~\ref{T:finite_interval}

\section{Remarks and conclusions}\label{S:remarks}

\begin{itemize}
\item In order to solve the exterior problem in terms of expansion in
spherical harmonics, we smoothly extended the $C_{0}^{\infty}(Z_{1})$
candidate function $b(t,\theta)$ to $C_{0}^{\infty}(\mathbb{R}\times S)$. Such
an extension has allowed us to obtain the expression (\ref{E:Rml}) for the
exterior Radon transform of the solution $v^{(b)}(t,x)$, with functions
$R_{l}^{\mathbf{m}}(t)$ computed as convolutions (\ref{E:varlimits}). Such a
smooth extension can be constructed rather arbitrarily, using the Borel's
lemma. However, as proven in Proposition \ref{T:finite_interval}, computation
of $R_{l}^{\mathbf{m}}(t)$ in the interval $t\in\lbrack-1,0]$ needed for
Theorem \ref{T:main} requires the knowledge of the values of $b(t,\theta)$
only on $C_{0}^{\infty}(Z_{1})$. This eliminates any effect of the freedom of
choice of the smooth extension of $b.$

\item The results of this paper are somewhat related to the results obtained
in \cite{finch2006range} and \cite{venky2023range}, where in odd dimensions certain symmetries (with respect to $t=1$)
were observed in the range of the wave operator and the spherical means
operator. The presence
of such symmetries serves as a necessary and sufficient condition for a
function to be in the range of the corresponding operator. Results of our
paper, however, pertain to functions defined on the half-time interval
$t\in\lbrack0,1]$. They hold in both even and odd dimensions.

\item We have formulated our range conditions for the wave operator. However, related
results can be obtained for the spherical means operator, by utilizing the
connection between the two problems, expressed through equations
(\ref{E:means}), (\ref{E:g_from_means}).
\end{itemize}

\section{Acknowledgments}

A significant part of this paper was written during the authors' participation
in the Programme "Rich and Nonlinear Tomography" held at the Isaac Newton
Institute (INI) for Mathematical Sciences. The authors are thankful to the INI and to Prof. V.
Krishnan, whose presentation at the Programme has alerted us about open
problems in this area of research. Both authors acknowledge partial support by
the NSF, through awards NSF/DMS 1816430 and 2007408 (the first author),\ and
NSF/DMS\ 2405348 (the second author).

\end{document}